\documentclass[12pt]{amsart}
\usepackage{amsmath}
\usepackage{amssymb}
\usepackage{amsthm}
\usepackage[final]{hyperref}
\usepackage{graphicx}
\newtheorem{thm}{Theorem}
\newtheorem{lem}[thm]{Lemma}

\theoremstyle{definition}

\theoremstyle{definition}

\begin{document}

\title[Generalized Black-Scholes Equations]{Local Well-Posedness for Generalized Seminlinear Black-Scholes Equations}

\author[D. da Silva]{Daniel Oliveira da Silva}
\address{Department of Mathematics \\
Nazarbayev University \\
Kabanbay Batyr Avenue 53 \\
010000 Astana \\
Republic of Kazakhstan}
\email{daniel.dasilva@nu.edu.kz}

\author[K. Igibayeva]{Kamilla Igibayeva}
\address{Department of Mathematics \\
Nazarbayev University \\
Kabanbay Batyr Avenue 53 \\
010000 Astana \\
Republic of Kazakhstan}
\email{kamilla.igibayeva@nu.edu.kz}

\author[A. Khoroshevskaya]{Adelina Khoroshevskaya}
\address{Department of Mathematics \\
Nazarbayev University \\
Kabanbay Batyr Avenue 53 \\
010000 Astana \\
Republic of Kazakhstan}
\email{adelina.khoroshevskaya@nu.edu.kz}

\author[Zh. Sakayeva]{Zhanna Sakayeva}
\address{Department of Mathematics \\
Nazarbayev University \\
Kabanbay Batyr Avenue 53 \\
010000 Astana \\
Republic of Kazakhstan}
\email{zhanna.sakayeva@nu.edu.kz}


\begin{abstract}
We consider some parabolic equations which are model problems for a variety of nonlinear generalizations to the Black-Scholes equation of mathematical finance.  In particular, we prove local well-posedness for the Cauchy problem with initial data in $L^{p}$ spaces.
\end{abstract}

\maketitle


\section{Introduction}

The original Black-Scholes equation, proposed by Black and Scholes in \cite{BS1973}, is the linear parabolic partial differential equation
\begin{equation}\label{bsequation}
V_{\tau} + \frac{\sigma^2}{2} S^2 V_{SS} + r S V_{S} - r V = 0.
\end{equation}
Here, $V(S,\tau): [0,\infty) \times [0,T] \longrightarrow \mathbb{R}$ is function which represents the monetary value of an \emph{option}, which is a contract to buy or sell an asset at a specified price, called the \emph{exercise price}, at a future time $T$.  The value $V$ is assumed to depend on the current time $\tau \leq T$, and the current price $S \geq 0$ of the underlying asset.  The constant $\sigma > 0$ represents the \emph{volatility}, which is a measure of the fluctuation of the value of the asset, while $r$ represents an interest rate for a bank account, in which investors could deposit their money with a guaranteed rate of return instead of buying the option.  This equation has become an important tool in the field of financial modeling, and more recently, the insurance industry \cite{O2014}.

In recent years, more sophisticated models of options pricing have appeared which have led to nonlinear generalizations of equation \eqref{bsequation}.  Some relevant examples are:
\begin{itemize}
\item The Barles-Soner Model \cite{BS1998}:
\[
V_{\tau} + \frac{\sigma^2}{2} S^2 V_{SS} + r S V_{S} - r V = \frac{\sigma^2}{2} S^2 V_{SS} G\left( e^{r(T-\tau)} a^2 S^2 V_{SS} \right),
\]
where $a$ is a parameter and $G$ is a function which is a solution to a certain ordinary differential equation.

\item The risk-adjusted Black-Scholes equation (RABS) \cite{K1998,JS2005}:
\begin{equation}\label{RABS}
V_{\tau} + \frac{\sigma^2}{2} S^2 \left( 1 - \mu |S V_{SS}|^{-2/3}V_{SS} \right) V_{SS} + r S V_{S} - r V = 0,
\end{equation}
with $\mu$ a parameter.

\item The \c Cetin-Jarrow-Protter (CJP) model \cite{CJP2004,FP2011}:
\begin{equation}\label{CJP}
V_{\tau} + \frac{\sigma^2}{2} S^2 V_{SS} \left( 1 + 2 \rho S V_{SS} \right) = 0,
\end{equation}
where $\rho$ is a parameter.

\item The Platen-Schweizer (PS) model \cite{PS1998,FP2011}:
\[
V_{\tau} + \frac{\sigma^2}{2}\frac{S^2}{(1-\rho S V_{SS})^2}V_{SS} = 0.
\]
\end{itemize}
These equations are fully nonlinear equations, which can often be quite difficult to study.

Due to the difficulty in analyzing fully nonlinear partial differential equations, many authors have proposed semilinear models, which are significantly easier to study.  Some examples of such semilinear models can be found in \cite{KW2009,LS2007,BKR2004,BKR2003, K1997}.  Since the appearance of these results, some authors have begun considering generalized Black-Scholes models.  One example of interest to us is the generalized model of Sowrirajan and Balachandran \cite{SB2010}, who considered generalized models of the form
\[
V_{\tau} + \frac{\sigma^2}{2} S^2 V_{SS} + r S V_{S} - r V = F(V)
\]
where $F$ is a Lipschitz function.  In the present work, we will first consider the problem studied by Sowrirajan and Balachandran, for which we will show local well-posedness for $L^{p}$ data by a simpler method involving a fixed-point argument.  Then, by applying the same methods, we will prove local well-posedness for more generalized Black-Scholes models which are based on the results listed above.  More precisely, we will first study the equation
\begin{equation}\label{equation1}
u_{\tau} + A S^{2} u_{SS} + B S u_{S} + C u = F_{1}(u)
\end{equation}
with constant coefficients $A,B,C \in \mathbb{R}$, and with nonlinearities $F$ which are Lipschitz functions.  

To motivate the second equation we will study, recall the risk-adjusted Black-Scholes equation \eqref{RABS}, and the \c Cetin-Jarrow-Protter model from equation \eqref{CJP}.  By introducing the new function $u = S V_{SS}$ as was done in \cite{JS2005}, it is easy to see that equations \eqref{RABS} and \eqref{CJP}  can be rewritten in a modified Black-Scholes form as
\[\begin{aligned}
& u_{\tau} + A_{1} S^{2}u_{SS} + B_{1} S u_{S} - B_{1} u = D_{1} S(u^{4/3})_{S} + E_{1} S^{2} (u^{4/3})_{SS}, \\
& u_{\tau} + A_{2} S^{2}u_{SS} + B_{2} S u_{S} - C_{2} u + D_{2} S (u^2)_{S} + E_{2} S^2 (u^{2})_{SS} = 0,
\end{aligned}\]
respectively, for appropriate constants $A_{i}, B_{i}, C_{i}, D_{i}$, and $E_{i}$.  Note that the coefficients on the left side of each equation differ from those in the original Black-Scholes equation \eqref{bsequation}.  Regardless of this, these two equations have a similar structure, namely in that they are both of the form
\[
u_{\tau} + A S^{2} u_{SS} + B S u_{SS} + C u = D S\left( F(u) \right)_{S} + E S^{2} \left( F(u) \right)_{SS}
\]
for some coefficients $A,B,C,D,E \in \mathbb{R}$ and where $F(u) = u^{p}$ for some exponent $p > 1$.  The presence of the second derivative in the right-hand side makes these equations quasilinear.  If we assume that $E = 0$ we then obtain the semilinear equation
\begin{equation}\label{equation2}
u_{\tau} + A S^{2} u_{SS} + B S u_{S} + C u = D S\left( F(u) \right)_{S}.
\end{equation}
This is the second equation which will be studied below.

To formulate our main results, we first rewrite equations \eqref{equation1} and \eqref{equation2} as nonlinear heat equations.  To do this, we use a standard change of variables which is used in the study of Black-Scholes equations.  Let
\[
x = \ln S, \quad \textrm{and} \quad t = A(T - \tau).
\]
A simple computation will show that, if $A \neq 0$, we may rewrite equations \eqref{equation1} and \eqref{equation2} as
\[\begin{aligned}
& u_{t} - u_{xx} = A_1 u_{x} + A_2 F_{1}(u), \\
& u_{t} - u_{xx} = B_1 u + B_2 (F_{2}(u))_{x},
\end{aligned}\]
respectively.  Note that the variable $x$ can now be negative, in contrast to the variable $S$.  In addition, the constants $A_i$ and $B_i$ are largely irrelevant to the analysis.  We will thus set them all to be equal to 1.

It should be noted that in applications of Black-Scholes equations, the interest is primarily in the final value problem; that is, we want to find a function $u(S,\tau)$ such that $u(S,T) = f(S)$ for some specified function $f$.  The change of variables above not only converts the equation into a heat-type equation, but it also converts the final value problem into an initial value problem, for which many standard tools apply.  Thus, the above change of variables for $\tau$ is not just an artificial change which puts the equation in a standard form.  With this observation in mind, we may now state the main results to be proved.

\begin{thm}\label{mainthm1}
Let $F_{1}: \mathbb{R} \longrightarrow \mathbb{R}$ be a Lipschitz function.  Then the Cauchy problem
\begin{equation}\label{maineq1}\begin{aligned}
& u_{t} - u_{xx} = u_{x} + F_{1}(u) \\
& u(x,0) = f (x)
\end{aligned}\end{equation}
is locally well-posed in $L^{p}$ for any $p \geq 1$.  That is, for all initial conditions $f \in L^{p}$, there is a $T > 0$ such that there exists a unique solution $u \in C([0,T]; L^{p})$ to equation \eqref{maineq1}.  Moreover, the solution map $f \mapsto u$ is continuous.
\end{thm}

\begin{thm}\label{mainthm2}
Let $F_{2}: \mathbb{R} \longrightarrow \mathbb{R}$ be a Lipschitz function.  Then the Cauchy problem
\begin{equation}\label{maineq2}\begin{aligned}
& u_{t} - u_{xx} = u + (F_{2}(u))_{x} \\
& u(x,0) = f (x)
\end{aligned}\end{equation}
is locally well-posed in $L^{p}$ for any $p \geq 1$.  That is, for all initial conditions $f \in L^{p}$, there is a $T > 0$ such that there exists a unique solution $u \in C([0,T]; L^{p})$ to equation \eqref{maineq2}.  Moreover, the solution map $f \mapsto u$ is continuous.
\end{thm}

\noindent The proofs of both of these theorems will involve a standard fixed-point argument.  In section \ref{prel}, we will introduce the necessary tools for the proofs.  Then, in sections \ref{proof1} and \ref{proof2}, we will use the results in section \ref{prel} to prove Theorems \ref{mainthm1} and \ref{mainthm2}, respectively.


\section{Preliminaries}\label{prel}

To set up the argument, we first recall some basic facts about the linear heat equation.  Consider the linear Cauchy problem
\[\begin{aligned}
& u_{t} - u_{xx} = G(x,t), \\
& u(0) = f.
\end{aligned}\]
for $x \in \mathbb{R}$, $t \geq 0$.  The solution to this problem is given by the well-known formula
\[
u(x,t) = \int_{\mathbb{R}} \Phi(x-y,t) f(y)\ dy + \int_{0}^{t} \int_{\mathbb{R}} \Phi(x-y, t-s) G(y,s)\ dy ds,
\]
where the function $\Phi(x,t)$ is the so-called \emph{heat kernel}, given by
\[
\Phi(x,t) = \begin{cases}
\displaystyle \frac{1}{\sqrt{4 \pi t}} e^{-x^2/4t} & \textrm{if } t > 0, \\
0 & \textrm{otherwise}.
\end{cases}
\]
To apply the solution formula above to a nonlinear problem of the form
\begin{equation}\label{gennlprob}\begin{aligned}
& u_{t} - u_{xx} = G(u,u_{x}), \\
& u(0) = f,
\end{aligned}\end{equation}
define a function $\Psi$ on $C([0,T];L^{p})$ by
\[
\Psi(u) = \int_{\mathbb{R}} \Phi(x-y,t) f(y)\ dy + \int_{0}^{t} \int_{\mathbb{R}} \Phi(x-y, t-s) G(y,s)\ dy ds,
\]
where $G(y,s) = G(u(y,s),u_{x}(y,s))$ here.  It is easy to see that the function $w = \Psi(u)$ satisfies
\[\begin{aligned}
& w_{t} - w_{xx} = G(u,u_{x}), \\
& w(0) = f,
\end{aligned}\]
If we can prove the existence of a fixed point $u$ for the operator $\Psi$, then this fixed point will satisfy the nonlinear problem in equation \eqref{gennlprob}.

The existence of such fixed points will follow from an application of the contraction mapping theorem.  Thus, we need only prove that $\Psi$ is a contraction on $C([0,T]; L^{p})$.  For this, we will need the following estimates on the heat kernel $\Phi$:
\begin{lem}\label{heatkernestimates}
The heat kernel $\Phi(x,t)$ satisfies the identities
\begin{itemize}
\item $\displaystyle \| \Phi(\cdot, t) \|_{L^{p}} = \frac{C_{p}}{(4 \pi t)^{\frac{1}{2}(1-1/p)}}$

\item $ \displaystyle \| \Phi_{x}(\cdot, t) \|_{L^{p}} = \frac{C_{p}}{(4 \pi t)^{1 - 1/2p}}$
\end{itemize}
for some constants $C_{p}$ and all $t > 0$.
\end{lem}
\noindent These can be proved by direct integration.

Next, we need the following estimates for the linear and nonlinear portion of the solutions:
\begin{lem}\label{solest}
Let $f \in L^{p}$ and $G \in C([0,T];L^{p}(\mathbb{R}))$.  Then the following estimates hold:
\begin{equation}\label{linest}
\left\| \int_{\mathbb{R}} \Phi(x-y,t) f(y)\ dy \right\|_{L^{\infty}_{t}L^{p}_{x}} \leq \| f \|_{L^{p}}
\end{equation}
\begin{equation}\label{nlest1}
\left\| \int_{0}^{t} \int_{\mathbb{R}} \Phi(x-y,t-s) G(y,s)\ dyds \right\|_{L^{\infty}_{t}L^{p}_{x}} \leq C T \| G \|_{L^{\infty}_{t}L^{p}_{x}}
\end{equation}
for some constant $C > 0$.  If we further assume that $G$ is differentiable, then
\begin{equation}\label{nlest2}
\left\| \int_{0}^{t} \int_{\mathbb{R}} \Phi(x-y,t-s) G_{x}(y,s)\ dyds \right\|_{L^{\infty}_{t}L^{p}_{x}} \leq C T^{1/2}\| G \|_{L^{\infty}_{t}L^{p}_{x}}
\end{equation}
for some constant $C > 0$.
\end{lem}
\begin{proof}
To begin the proof, we first observe that
\[
\int_{\mathbb{R}} \Phi(x-y,t) f(y)\ dy = \Phi(\cdot, t) * f (x),
\]
where $*$ denotes convolution in the spatial variable.  Thus
\[
\left\|\int_{\mathbb{R}} \Phi(x-y,t) f(y)\ dy\right\|_{L^{p}_{x}} = \left\|\Phi(\cdot, t) * f (x)\right\|_{L^{p}_{x}}.
\]
By Young's inequality, we have
\[
\left\|\Phi(\cdot, t) * f (x)\right\|_{L^{p}_{x}} \leq \| \Phi(\cdot, t) \|_{L^{1}_{x}} \| f \|_{L^{p}}.
\]
By Lemma \ref{heatkernestimates}, this reduces to
\[
\left\|\Phi(\cdot, t) * f (x)\right\|_{L^{p}_{x}} \leq \| f \|_{L^{p}}.
\]
The desired result follows by taking the $L^{\infty}$ norm in the time variable.

To prove equation \eqref{nlest1}, we first write
\[
\int_{0}^{t} \int_{\mathbb{R}} \Phi(x-y,t-s) G(y,s)\ dyds = \int_{0}^{t} \Phi(\cdot,t-s) * G(\cdot,s)(x)\ ds
\]
By Minkowski's inequality, Young's inequality and Lemma \ref{heatkernestimates}, we have
\[\begin{aligned}
& \left\| \int_{0}^{t} \int_{\mathbb{R}} \Phi(x-y,t-s) G(y,s)\ dyds \right\|_{L^{p}_{x}} \\
& \qquad \qquad \qquad \qquad  \leq  \int_{0}^{t} \| \Phi(\cdot,t-s)\|_{L^{1}_{x}} \| G(\cdot,s) \|_{L^{p}_{x}}\ ds \\
& \qquad \qquad \qquad \qquad  \leq  \int_{0}^{t} \| G(\cdot,s) \|_{L^{p}_{x}}\ ds \\
& \qquad \qquad \qquad \qquad \leq  t \| G \|_{L^{\infty}_{t}L^{p}_{x}}.
\end{aligned}\]
Taking the $L^{\infty}$ norm in time yields
\[
\left\| \int_{0}^{t} \int_{\mathbb{R}} \Phi(x-y,t-s) G(y,s)\ dyds \right\|_{L^{\infty}_{t}L^{p}_{x}} \leq T \| G \|_{L^{\infty}_{t} L^{p}_{x}},
\]
as desired.

For the proof of estimate \eqref{nlest2}, we observe that
\[\begin{aligned}
\int_{\mathbb{R}} \Phi(x-y,t-s) G_{x}(y,s)\ dy & = \Phi(\cdot,t-s) * G_{x}(\cdot,s)(x) \\
& = \Phi_{x}(\cdot,t-s) * G(\cdot,s)(x)
\end{aligned}\]
by standard properties of convolutions.  If we again apply Minkowski's inequality, Young's inequality, and Lemma \ref{heatkernestimates}, we obtain
\[
\left\| \int_{0}^{t}\int_{\mathbb{R}} \Phi(x-y,t-s) G_{x}(y,s)\ dy \right\|_{L^{p}_{x}} \leq C t^{1/2}\| G \|_{L^{\infty}_{t}L^{p}_{x}}.
\]
Taking the $L^{\infty}$ norm in time yields equation \eqref{nlest2}, as desired.
\end{proof}


\section{Proof of Theorem \ref{mainthm1}}\label{proof1}

In this section, we consider equation \eqref{gennlprob} in the case 
\[
G(u,u_{x}) = u_{x} + F_1(u)
\] 
for some Lipschitz function $F_1$ with $F_1(0) = 0$.  By Lemma \ref{solest}, we have that \[\Psi: C([0,T];L^{p}(\mathbb{R})) \longrightarrow C([0,T];L^{p}(\mathbb{R})).\]  Thus, we need only prove an estimate of the form
\[
\| \Psi(u) - \Psi(v) \|_{L^{\infty}_{t}L^{p}_{x}} \leq C \| u - v \|_{L^{\infty}_{t}L^{p}_{x}}
\]
for some constant $0 < C < 1$.  To obtain such an estimate, we first observe that, for fixed $f \in L^{p}$, we have
\[\begin{aligned}
\Psi(u) - \Psi(v) & = \int_0^t \int_{\mathbb{R}} \Phi(x-y,t-s)[u_{x}(y,s)- v_{x}(y,s)]dyds \\
& + \int_0^t \int_{\mathbb{R}} \Phi(x-y,t-s)[F_{1}(u(y,s))-F_{1}(v(y,s))]dyds.
\end{aligned}\]
If we apply the estimate in equation \eqref{nlest2}, we see that
\[\begin{aligned}
& \left\| \int_0^t \int_{\mathbb{R}} \Phi(x-y,t-s)[u_{x}(y,s)- v_{x}(y,s)]dyds \right\|_{L^{\infty}_{t} L^{p}_{x}} \\
& \qquad \qquad \qquad \qquad \qquad \qquad \qquad \qquad \leq C T^{1/2} \| u - v \|_{L^{\infty}_{t} L^{p}_{x}}.
\end{aligned}\]
If we apply the estimate in equation \eqref{nlest1}, it is easy to see that
\[\begin{aligned}
& \left\| \int_0^t \int_{\mathbb{R}} \Phi(x-y,t-s)[F_{1}(u(y,s))-F_{1}(v(y,s))]dyds \right\|_{L^{\infty}_{t}L^{p}_{x}} \\
& \qquad \qquad \qquad \qquad \qquad \qquad \qquad \qquad \leq C T \| F_{1}(u) - F_{1}(v) \|_{L^{\infty}_{t} L^{p}_{x}}
\end{aligned}\]
Since $F_{1}$ is Lipschitz, it follows that
\[
\| F_{1}(u) - F_{1}(v) \|_{L^{\infty}_{t} L^{p}_{x}} \leq C \| u - v \|_{L^{\infty}_{t} L^{p}_{x}}
\]
Putting all this together, we have shown that
\[
\| \Psi(u) - \Psi(v) \|_{L^{\infty}_{t}L^{p}_{x}} \leq C(T+T^{1/2})\| u - v \|_{L^{\infty}_{t}L^{p}_{x}}
\]
for some constant $C > 0$.  If we choose $T$ small enough so that
\[
C(T+T^{1/2}) < 1,
\]
we will then have that $\Psi$ is a contraction on $C([0,T];L^{p}(\mathbb{R}))$.

Next, we prove that the solution map $f \mapsto u$ is continuous.  Let $f, g \in L^{p}$, and let $u_{f}$ and $u_{g}$ be the corresponding fixed points.  By applying estimates \eqref{linest}, \eqref{nlest1}, and \eqref{nlest2}, we may obtain
\[\begin{aligned}
\| u_{f} - u_{g} \|_{L^{\infty}_{t}L^{p}_{x}} & = \| \Psi(u_{f}) - \Psi(u_{g}) \|_{L^{\infty}_{t}L^{p}_{x}} \\
& \leq \| f - g \|_{L^{p}_{x}} + C ( T + T^{1/2} ) \| u_{f} - u_{g} \|_{L^{\infty}_{t}L^{p}_{x}},
\end{aligned}\]
which implies that
\[
\| u_{f} - u_{g} \|_{L^{\infty}_{t} L^{p}_{x}} \leq \frac{1}{1 - C( T + T^{1/2} )} \| f - g \|_{L^{p}_{x}}.
\]
It follows that, for $T$ as chosen above, the solution map is continuous.  In fact, the computation above implies the solution map is Lipschitz.


\section{Proof of Theorem \ref{mainthm2}}\label{proof2}
We now consider equation \eqref{gennlprob} for
\[
G(u,u_{x}) = u + (F_{2}(u))_{x},
\]
where $F_{2}$ is a differentiable Lipschitz function.  As before, it suffices to prove an estimate of the form
\[
\| \Psi(u) - \Psi(v) \|_{L^{\infty}_{t} L^{p}_{x}} \leq C \| u - v \|_{L^{\infty}_{t} L^{p}_{x}}
\]
for some $0 < C < 1$.  Proceeding exactly as we did in the proof of Theorem \ref{mainthm1}, we first observe that equation \eqref{nlest1} implies that
\[
\left\| \int_0^t \int_{\mathbb{R}} \Phi(x-y,t-s)[u(y,s)- v(y,s)]dyds \right\|_{L^{\infty}_{t} L^{p}_{x}}  \leq C T \| u - v \|_{L^{\infty}_{t} L^{p}_{x}}.
\]
Similarly, estimate \eqref{nlest2} implies that
\[\begin{aligned}
& \left\| \int_0^t \int_{\mathbb{R}} \Phi(x-y,t-s)[(F_{2}(u(y,s))-F_{2}(v(y,s)))_{x}]dyds \right\|_{L^{\infty}_{t}L^{p}_{x}} \\
& \qquad \qquad \qquad \qquad \qquad \qquad \qquad \qquad \leq C T^{1/2} \| F_{2}(u) - F_{2}(v) \|_{L^{\infty}_{t} L^{p}_{x}}.
\end{aligned}\]
Applying the assumption that $F_{2}$ is Lipschitz, we see that
\[
\| F_{2}(u) - F_{2}(v) \|_{L^{\infty}_{t} L^{p}_{x}} \leq \| u - v \|_{L^{\infty}_{t} L^{p}_{x}}.
\]
Combining these facts together, we have that
\[
\| \Psi(u) - \Psi(v) \|_{L^{\infty}_{t} L^{p}_{x}} \leq C(T + T^{1/2}) \| u - v \|_{L^{\infty}_{t} L^{p}_{x}}.
\]
If we again choose $T$ so that
\[
C(T+T^{1/2}) < 1,
\]
then we again have that $\Psi$ is a contraction.  The existence of a unique fixed point then follows from the Banach fixed point theorem.  The proof that the solution map is Lipschitz (and therefore continuous) is analogous to that of Theorem \ref{mainthm1}, and so we will be omitted here.


\bibliographystyle{amsplain}
\bibliography{database}

\end{document}